\documentclass[reqno]{amsart}

    \usepackage{float}
    \usepackage{amsmath}
    \usepackage{graphicx}
    \usepackage{latexsym}
    \usepackage{amsfonts}
    \usepackage{amssymb}
    \usepackage{verbatim}
    \usepackage{color}
    \theoremstyle{plain}
    \newtheorem{theorem}{Theorem}
    
    \newtheorem{lemma}[theorem]{Lemma}
    
    \theoremstyle{definition}
    \newtheorem{assumption}{Assumption}

    \newtheorem*{remark*}{Remark}

    \newcommand{\pr}{\mathbf P}
    \newcommand{\e}{\mathbf E}

\begin{document}
    \title[Lower deviations for branching processes with immigration]
    {Lower deviations for branching processes with immigration}
    \thanks{Funded by the Deutsche Forschungsgemeinschaft (DFG, German Research Foundation) – Project-ID 317210226 – SFB 1283}
    \author[Sharipov]{Sadillo Sharipov}
    \address{V.I.Romanovskiy Institute of Mathematics, Tashkent, Uzbekistan}
    \email{sadi.sharipov@yahoo.com}

    \author[Wachtel]{Vitali Wachtel}
    \address{Faculty of Mathematics, Bielefeld University, Germany}
    \email{wachtel@math.uni-bielefeld.de}

\begin{abstract}
Let $\{Y_{n}$, $n \geq 1\}$ be a critical branching process with immigration having finite variance for the offspring number of particles and finite mean for the immigrating number of particles. In this paper, we study lower deviation probabilities for $Y_{n}$. More precisely, assuming that $k,n \to \infty$ such that $k=o\left(n \right)$, we investigate the asymptotics of $\pr\left(Y_{n} \leq k \right)$ and $\pr\left(Y_{n} = k \right)$. Our results clarify the role of the moment conditions in the local limit theorem for $Y_n$
proven in Mellein~\cite{Mellein82}.
\end{abstract}


    \keywords{Branching process, local limit theorem, lower deviations}
    \subjclass{Primary 60J80; Secondary 60F10}
    \maketitle

\section{Introduction, main results and discussion}
Let $\{Y_n\}$ be a branching process with offspring distribution
$\{p_k\}$ and with immigration governed by the distribution $\{q_k\}$.
A mathematically rigorous description can be as follows. Let
$\{\xi_{n,j}\}$ be independent random variables with distribution
$\{p_k\}$ and let $\{\eta_n\}$ be independent random variables with distribution $\{q_k\}$. We assume also that the sequences $\{\xi_{n,j}\}$ and $\{\eta_n\}$ are independent from each other. Then the process $\{Y_n\}$ can be defined by the following recursive equation
\begin{equation}
\label{eq:Y-def}
Y_{n+1}=\sum_{j=1}^{Y_n}\xi_{n+1,j}+\eta_{n+1},\quad n\geq 0.
\end{equation}
The starting point $Y_{0}$ can be deterministic or random; one only assumes that $Y_{0}$ is independent of the sequences $\{\xi_{n,j}\}$ and $\{\eta_n\}$.

In this paper we shall consider only critical processes, that is,
\begin{equation}
\label{eq:crit}
\sum_{k=1}^{\infty} kp_{k}=1.
\end{equation}
Furthermore, we shall always assume that
\begin{equation}
\label{eq:moments1}
B:=\sum_{k=1}^\infty k\left(k-1\right)p_{k}\in\left(0,\infty\right)
\quad\text{and}\quad
\lambda:=\sum_{k=1}^{\infty} kq_{k}\in\left(0,\infty\right).
\end{equation}
It is worth mentioning that the strict positivity of $\sum_{k=1}^{\infty} kq_{k}$ means that the immigration component of the process is a non-degenerate and, consequently, standard Galton-Watson processes are excluded.

Seneta~\cite{Seneta70} has shown that if \eqref{eq:crit} and \eqref{eq:moments1} are valid then $\frac{2Y_{n}}{Bn}$ converges weakly towards a $\Gamma$-distribution. More precisely, for every $x>0$ one has
\begin{equation}
\label{eq:gamma-limit}
\lim_{n\to\infty}\pr\left(\frac{2Y_{n}}{Bn}\leq x\right)
=\frac{1}{\Gamma\left(\gamma\right)}\int_0^x u^{\gamma-1}e^{-u}du,
\end{equation}
where
$$
\gamma:=\frac{2\sum_{k=1}^{\infty} kq_{k}}{B}.
$$
Mellein~\cite{Mellein82} has proven the corresponding local limit theorem. Assuming that
\begin{equation}
\label{eq:moments2}
\sum_{k=1}^\infty \left(k^2\log k\right) p_{k}<\infty
\quad\text{and}\quad
\sum_{k=1}^\infty \left(k\log k\right) q_{k}<\infty
\end{equation}
he has shown that, for every fixed $i$,
\begin{equation}
\label{eq:loc-limit}
\pr\left(Y_{n}=k|Y_{n}=i\right) \sim
\frac{1}{\Gamma\left(\gamma\right)}
\left(\frac{2}{B}\right)^{\gamma}
\frac{k^{\gamma-1}}{n^\gamma}
e^{-2k/Bn}
\end{equation}
as $n,k\to\infty$ and $k/n$ remains bounded.

Since the moments in \eqref{eq:moments2} do not show up in the asymptotics \eqref{eq:loc-limit}, it is natural to ask whether that conditions are needed for the local limit theorem. This question is our main motivation for considering lower deviation probabilities for $Y_{n}$. More precisely, we are going to determine the asymptotic behaviour of the probabilities $\pr\left(Y_{n}\leq k\right)$ and $\pr\left(Y_{n}=k\right)$ in the case when $k\to\infty$ but $k=o\left(n\right)$ for processes satisfying \eqref{eq:crit} and \eqref{eq:moments1} only. Comparing asymptotics for local probabilities with the corresponding values of the limiting density, we can then conclude whether the extra condition \eqref{eq:moments2} is really needed for the validity of \eqref{eq:loc-limit}.
\begin{theorem}
\label{thm:main}
Assume that \eqref{eq:crit} and \eqref{eq:moments1} hold. Then there exists a slowly varying function $L\left(x\right)$ such that, for every fixed $i\ge0$,
\begin{equation}
\label{eq:main1}
\pr\left(Y_{n}\leq k|Y_0=i\right)\sim
\frac{1}{\Gamma(\gamma+1)}
\left(\frac{2}{B}\right)^{\gamma}
\frac{k^\gamma L\left(k\right)}{n^\gamma L\left(n\right)}
\end{equation}
for $k\to\infty$ and $k=o\left(n\right)$. If, additionally, $\{p_{k}\}$ and
$\{q_{k}\}$ are aperiodic, then
\begin{equation}
\label{eq:main2}
\pr\left(Y_{n}=k|Y_0=i\right)\sim
\frac{1}{\Gamma(\gamma)}
\left(\frac{2}{B}\right)^{\gamma}
\frac{k^{\gamma-1} L\left(k\right)}{n^{\gamma} L\left(n\right)}.
\end{equation}
For every fixed $k$ there exists $\mu_{k}$ such that
\begin{equation}
\label{eq:main3}
n^{\gamma} L\left(n\right)
\pr\left(Y_{n}=k|Y_0=i\right)\to\mu_{k}.
\end{equation}
The function $L\left(x\right)$ converges to a positive constant provided that \eqref{eq:moments2} holds.
If $\sum_{k=1}^\infty \left(k^{2}\log k\right) p_{k}<\infty$ and
$\sum_{k=1}^\infty \left(k\log k\right) q_{k}=\infty$ then
$L\left(n\right)$ tends to infinity. Moreover, if $\sum_{k=1}^\infty \left(k^{2}\log k\right) p_{k}=\infty$ and
$\sum_{k=1}^\infty \left(k\log k\right) q_{k}<\infty$
then $L(n)$ tends to zero.
\end{theorem}
This result shows that \eqref{eq:loc-limit} can fail if the moment conditions in \eqref{eq:moments2} are not valid. This is quite different from the case of critical Galton-Watson processes. Nagaev and Wachtel~\cite{NV06} have shown that the existence of the second moment is sufficient for the local limit theorem.

We conjecture that the domain of lower deviations is the only zone where one needs \eqref{eq:moments2} for the local asymptotics \eqref{eq:loc-limit}. In other words, we believe that the following statement holds: for every fixed $\varepsilon>0$ one has
\begin{equation}
\label{eq:conjecture}
\sup_{k\geq\varepsilon n}\left|n\pr\left(Y_{n}=k\right)
-\frac{1}{\Gamma\left(\gamma\right)}\left(\frac{2k}{Bn}\right)^{\gamma-1}e^{-2k/Bn}\right|\to 0.
\end{equation}
Noting that $u^{\gamma-1}e^{-u}\to 0$ as $u\to 0$ for $\gamma>1$ and combining \eqref{eq:conjecture} and  \eqref{eq:main2}, one gets for $\gamma>1$ the most standard version of the local limit theorem:
$$
\sup_{k\ge0}\left|n\pr\left(Y_{n}=k\right)
-\frac{1}{\Gamma\left(\gamma\right)}\left(\frac{2k}{Bn}\right)^{\gamma-1}e^{-2k/Bn}\right|\to 0.
$$
(One should notice that this relation is not stronger than \eqref{eq:loc-limit}, because for $k=o\left(n\right)$ one gets only
$\pr\left(Y_{n}=k\right)=o\left(n^{-1}\right)$.)

Pakes~\cite{Pakes72} has shown \eqref{eq:main3} under the condition \eqref{eq:moments2}. Therefore, Theorem~\ref{thm:main} generalizes the results of Mellein~\cite{Mellein82} and Pakes~\cite{Pakes72} to the whole class of processes satisfying \eqref{eq:moments1}. Our Theorem~\ref{thm:main} can be used to derive in a rather simple way asymptotics for the so-called harmonic moments $\e[Y_n^{-r};Y_n>0]$ and. This has been done, by other means, in Li and Zhang~\cite{LiZhang21}.

Our approach to lower deviation probabilities differs from that in the papers by Mellein and by Pakes. We determine first an optimal strategy which leads to atypically small values of $Y_{n}$. This allows us to reduce a problem on small deviations to normal deviations, where one uses the known in the literature results. To describe that optimal strategy, we introduce some additional notation. Let $\{Z_{n}^{(i)}\}$,
$i\ge1$ be a sequence of independent Galton-Watson processes with offspring distribution $\{p_{k}\}$ and with $\pr\left(Z_{0}^{\left(i\right)}=k\right)=q_{k}$, $k\ge 0$. It is then immediate from \eqref{eq:Y-def} that if $Y_{0}=0$
then
\begin{equation}
\label{eq:repre}
Y_{n}=\sum_{i=1}^{n} Z^{\left(i\right)}_{n-i}\quad\text{in distribution}.
\end{equation}
Define now
\begin{equation}
\label{eq:theta}
\theta_{n}:=\inf\{i\leq n: Z^{\left(i\right)}_{n-i}>0\}.
\end{equation}
In words, $\theta_{n}$ is the first generation where immigrants have descendants at time $n$. We know from \eqref{eq:gamma-limit} that $Y_{n}$ is typically of size $n$ and that a Galton-Watson process $Z_{n}$ conditioned on non-extinction is also of order $n$. Thus, it is quite plausible to assume that if we want to have only $k$ particles in the $n$th generation, then $\theta_{n}$ should be such that $n-\theta_{n}$ is of order $k$. In the course of the proof of Theorem~\ref{thm:main}, we show that this strategy is indeed optimal.

A strategy when one keeps the size of the process as small as possible up to an appropriate time moment is quite standard for branching processes. This type of behaviour is also optimal for lower deviations of supercritical Galton-Watson processes; see, for example,
\cite{FW07}.

The paper is organized as follows. In Section 2, we provide some results about upper bounds for lower deviation probabilities.
Section 3 contains estimates for the concentration function, and Section 3 gathers the proof of the main theorem. 

 \section{Analysis of generating functions and some preliminary probabilistic estimates.}
 Define $f\left(s\right):=\sum_{k=0}^\infty p_{k}s^{k}$ and
 $h\left(s\right):=\sum_{k=0}^\infty q_ks^k$. In other words, $f\left(s\right)$ is the offspring generating function of the processes $\{Y_{n}\}$ and 
 $\{Z_{n}^{\left(i\right)}\}$, and $h\left(s\right)$
 is the generating function of the number of immigrants in $\{Y_{n}\}$ and of
 $Z_{0}^{\left(i\right)}$. Let $f_{n}\left(s\right)$ denote the $n$th iteration of $f$.
 
 In this section we shall {\it always} assume that $Y_0=0$.
 
 It is immediate from \eqref{eq:repre} that
$$
H_n(s):=\e\left[s^{Y_n}\right]
=\prod_{k=0}^{n-1}h(f_k(s)).
$$

 For the random variable $\theta_{n}$ defined in \eqref{eq:theta} we then have for every $l\leq n$
 \begin{equation*}
 \pr\left(\theta_{n}>l\right)=\prod_{i=1}^l\pr\left(Z_{n-i}^{(i)}=0\right)
 =\prod_{i=1}^{l} h\left(f_{n-i}\left(0\right)\right)
 =\prod_{k=n-l}^{n-1}h\left(f_{k}\left(0\right)\right)
 \end{equation*}
and
\begin{equation*}
 \pr\left(\theta_n=l\right)=\pr\left(\theta_n>l-1\right)-\pr\left(\theta_n>l\right)
 =\left(1-h\left(f_{n-l}\left(0\right)\right)\right)\prod_{k=n-l+1}^{n-1}h\left(f_{k}\left(0\right)\right).
 \end{equation*}
 Define
 $$
 F\left(0\right):=1\quad\text{and}\quad
 F\left(n\right):=\prod_{k=0}^{n-1}h(f_k(0)),\ n\geq 1.
 $$
 Then
 \begin{equation}
 \label{eq:theta-prob1}
 \pr(\theta_n>l)=\frac{F\left(n\right)}{F\left(n-l\right)}
 \end{equation}
 and
 \begin{equation}
 \label{eq:theta-prob2}
 \pr\left(\theta_n=l\right)=\left(1-h\left(f_{n-l}(0)\right)\right)\frac{F(n)}{F(n-l+1)}
 \end{equation}
 for every $l\le n$.\\
 We begin with a simple lemma, which will be used later in the analysis of the asymptotic behaviour of the sequence $F(n)$.
\begin{lemma}
\label{lemma2}
For $\left|s \right|\leq 1$, it holds that
\begin{equation}
\label{eq:15}
 \log h\left(s\right)=-h'\left(1\right)\left(1-s\right)+o\left(s-1\right), \ \ s\to 1.
\end{equation}
\end{lemma}
\begin{proof}
It follows from the condition $\lambda<\infty$ that
$$
h\left(s\right)=1+h'\left(1\right)\left(s-1\right)+o\left(s-1\right).
$$
Now the application of equality $\log \left(1+x\right)=x+O\left(x^{2}\right)$ gives us
$$
\log h\left(s\right)=\log \left(1+h'\left(1\right)\left(s-1\right)+o\left(s-1\right)\right)=
$$
$$
=h'\left(1\right)\left(s-1\right)+o\left(s-1\right)+O\left(\left(h'\left(1\right)\left(s-1\right)+o\left(s-1\right)\right)^{2} \right)
$$
which concludes the proof of the lemma.
\end{proof}
The following lemma describes the asymptotic behaviour of $F(n)$.
\begin{lemma}
\label{lemma3}
Assume that \eqref{eq:moments1} holds. Then the sequence
$L\left(n\right):=\left(n^\gamma F\left(n\right)\right)^{-1}$ is slowly varying.
\end{lemma}
\begin{proof}
According to the representation theorem for slowly varying functions, it suffices to represent $L\left(n\right)$ in the following form:
\begin{equation}
\label{eq:16}
 L\left(n\right)=c\exp\left\{\int_{1}^{n}\frac{\tau\left(y\right)}{y}dy \right\},
\end{equation}
where $\tau\left(y\right) \to 0$ as $y \to \infty$, and $c>0$.\\
We infer from \eqref{eq:15} that
\begin{equation}
\label{eq:17}
 \log h\left(s\right)=-h'\left(s\right)\left(1-s\right)+\left(s-1\right)\alpha\left(s\right),
\end{equation}
where $\alpha\left(s\right)\to 0$ as $s \to 1$.\\
Due to the condition $B<\infty$, 
\begin{equation}
\label{eq:f-iterations}
1-f_{j}\left(0\right)=\frac{2}{Bj}+\frac{\epsilon_{j}}{j}, 
\end{equation}
where $\epsilon_{j} \to 0$ as $j \to \infty$.
Combining that with \eqref{eq:17}, we have
$$
\begin{aligned}
n^{\gamma} F\left(n\right) & = n^{\gamma}\exp\left\{\sum_{j=0}^{n-1}\log h\left(f_{j}\left(0\right)\right) \right\}\\
& = n^{\gamma}\exp\left\{-h'\left(1\right)\sum_{j=0}^{n-1}\left(1-f_{j}\left(0\right) \right)-\sum_{j=1}^{n} \left(1-f_{j}\left(0\right) \right)\alpha\left(f_{j}\left(0\right)\right) \right\} \\
& = n^{\gamma}\exp\left\{-\frac{2h'\left(1\right)}{B}\sum_{j=1}^{n}\frac{1}{j} \right\} \exp\left\{\sum_{j=1}^{n}\frac{-h'\left(1\right)\epsilon_{j}}{j}-\frac{\alpha\left(f_{j}\left(0\right)\right)}{j} \right\}\\
& = n^{\gamma}\exp\left\{-\gamma \ln n \right\} \exp\left\{-\sum_{j=1}^{n}\frac{\left(h'\left(1\right)\epsilon_{j}+\alpha\left(f_{j}\left(0\right)\right)\right)}{j} \right\} \\
& = \exp\left\{-\sum_{j=1}^{n}\frac{\left(h'\left(1\right)\epsilon_{j}+\alpha\left(f_{j}\left(0\right)\right)\right)}{j} \right\}.
\end{aligned}
$$
It can be easily seen that the function $L\left(n\right)$ has the representation of the form \eqref{eq:16} with
$\tau\left(y\right):=h'\left(1\right)\epsilon_{n}+\alpha\left(f_{n}\left(0\right)\right)$ for $y\in\left[n-1,n\right]$. This completes the proof of Lemma \ref{lemma3}.
\end{proof}
\begin{lemma}
\label{lemma4}
Assume that \eqref{eq:moments1} holds. Then, as $k,n \to \infty$,
$$ \prod_{j=k}^{n-1} h\left(f_{j}\left(0\right) \right) \sim \left(\frac{k}{n} \right)^{\gamma} \frac{L\left( k\right)}{L\left(n\right)}.$$
\end{lemma}
\begin{proof}
This is an immediate consequence of Lemma \ref{lemma3}.
\end{proof}

The following statement describes the asymptotics of $L\left(n\right)$, which depends on the fulfillment of only one condition in (5).
\begin{lemma}
\label{lemma6}
The following statements hold true.
\begin{itemize}
\item[(a)] Assume that $\sum_{k=1}^\infty \left(k^{2}\log k\right) p_{k}<\infty$ and
$\sum_{k=1}^\infty \left(k\log k\right) q_{k}=\infty$. Then
\begin{equation}
\label{eq:23}
 L\left(n\right) \to +\infty\quad \text {as } n\to \infty.
\end{equation}
\item[(b)] Assume that $\sum_{k=1}^\infty \left(k^{2}\log k\right) p_{k}=\infty$ and
$\sum_{k=1}^\infty \left(k\log k\right) q_{k}<\infty$. Then
\begin{equation}
\label{eq:24}
 L\left(n\right) \to 0 \quad \text {as } n\to \infty.
\end{equation}
\end{itemize}
\end{lemma}
\begin{proof}
First, observe that
\begin{equation}
\label{eq:25}
 1-h\left(s\right)= \left(1-s\right) h'\left(1\right)-\left(1-s\right)^{2} \sum_{j=0}^{\infty}s^{j} \sum_{i=j+1}^{\infty}R_{i},
\end{equation}
where
$ R_{j}=\sum_{k=j+1}^{\infty}q_{k}. $

Using the expansion $\log(1+x)=x+O\left(x^{2}\right)$ and noting that
$h\left(f_{j}\left(0\right)-1\right)^{2} \le Cj^{-2}$ for any process satisfying \eqref{eq:moments1}, we conclude that
\begin{align}
\label{eq:F1}
\nonumber
F\left(n \right)
& =\exp\left\{\sum_{j=0}^{n-1}\log h\left(f_{j}\left(0\right)\right) \right\} \\
\nonumber
& = \exp\left\{\sum_{j=0}^{n-1}\left(h\left(f_{j}\left(0\right) \right)-1 \right)+O\left(h\left(f_{j}\left(0\right)\right)-1 \right)^{2} \right\} \\
&  \sim C_{0} \exp\left\{-\sum_{j=1}^{n-1} \left(1-h\left(f_{j}\left(0\right)\right) \right) \right\}.
\end{align}
Applying now the representation \eqref{eq:25}, we have
\begin{align}
\label{eq:decomp}
\nonumber
& \sum_{j=1}^{n-1}\left(1-h\left(f_{j}\left(0\right)\right) \right) \\
&\hspace{1cm}=h'(1)\sum_{j=1}^{n-1}(1-f_{j}(0))
- \sum_{j=1}^{n-1}\left(1-f_{j}\left(0\right)\right)^{2} \sum_{l=0}^{\infty}f_{j}^{l}
\left(0\right) \sum_{i=l+1}^{\infty}R_{i}.
\end{align}
If $\sum_{k\ge1}\left(k^{2}\log k\right)p_{k}$ is finite then, due to Lemma~8 in \cite{KestenNeySpitzer67},
$$
\sum_{j=1}^{n-1}(1-f_{j}(0))
=\sum_{j=1}^{n-1}\frac{2}{Bj}+C+o\left(1\right)
\quad\text{as }n\to\infty.
$$
This implies that
$$
F\left(n\right)=n^{-\gamma}L\left(n\right),
$$
where
$$
L\left(n\right) = C_{1}\exp\left\{\sum_{j=1}^{n-1}\left(1-f_{j}\left(0\right)\right)^{2} \sum_{l=0}^{\infty}f_{j}^{l}
\left(0\right) \sum_{i=l+1}^{\infty}R_{i}\right\}.
$$
It is now obvious that \eqref{eq:23} is equivalent to
\begin{equation}
\label{eq:26}
\lim_{n\to\infty}\sum_{j=1}^{n-1}\left(1-f_{j}\left(0\right)\right)^{2} \sum_{l=0}^{\infty}f_{j}^{l}
\left(0\right) \sum_{i=l+1}^{\infty}R_{i} = \infty.
\end{equation}
Set $\overline{R_{l}}:=\sum_{i=l+1}^{\infty}R_{i}$.
Using \eqref{eq:f-iterations} and interchanging the order of summation, we get
$$
\begin{aligned}
& \sum_{j=1}^{n-1}\left(1-f_{j}\left(0\right)\right)^{2} \sum_{l=0}^{\infty}f_{j}^{l}
\left(0\right) \sum_{i=l+1}^{\infty}R_{i} \\
& \geq c^{2} \sum_{j=1}^{n-1}\frac{1}{j^{2}}\sum_{l=1}^{j}\left(1-\frac{c}{j} \right)^{l}\overline{R_{l}}
\ge c_{1} \sum_{l=1}^{n-1} \overline{R_{l}} \sum_{j=l}^{n-1}\frac{1}{j^{2}} \\
& \geq \frac{c_{1}}{4} \sum_{l=1}^{n/2} \overline{R_{l}} \frac{1}{l}
\ge\frac{c_{1}}{4} \sum_{l=1}^{n/2}\frac{1}{l}\sum_{i=l+1}^{n/2} R_{i} \\
&= \frac{c_{1}}{4} \sum_{i=l+1}^{n/2 } R_{i} \sum_{l=1}^{i-1}\frac{1}{l} \geq c_{2} \sum_{i=2}^{n/2} R_{i}\log i
\to \infty,
\end{aligned}
$$
where in the last step we used that $\sum_{k=1}^{\infty}\left(k \log k\right)q_{k}= \infty$. This proves \eqref{eq:26} and thus \eqref{eq:23}.

We now turn to \eqref{eq:24}.
If $\sum_{k\ge 1}\left(k\log k\right)q_{k} $ is finite then, using the arguments from the proof of Lemma 2.1 in \cite{Mellein82}, we have
\begin{align*}
\sum_{j=1}^{\infty}\left(1-f_{j}\left(0\right)\right)^{2} \sum_{l=0}^{\infty}f_{j}^{l}
\left(0\right) \sum_{i=l+1}^{\infty}R_{i}
<\infty.
\end{align*}
Combining this with \eqref{eq:F1} and with \eqref{eq:decomp}, we establish
\begin{align}
\label{eq:F2}
F(n)\sim C_{2} \exp\left\{-h'(1)\sum_{j=1}^{n-1}\left(1-f_{j}(0)\right)\right\}
\end{align}

It is known, see Lemma 3 in \cite{NV06}, that
$$
1+\frac{Bn}{2}-\frac{1}{1-f_{n}\left(0\right)}=\sum_{k=0}^{n-1}\delta\left(f_{k}\left(0\right)\right)=o\left(n\right),
$$
where
$$
\delta\left(s\right)= \frac{\varepsilon\left(s\right)-\alpha\left(1-s \right)\left(\alpha-\varepsilon\left(s\right) \right)}
{1-\left(1-s \right)\left(\alpha-\varepsilon\left(s\right) \right)}, \ \ \alpha=\frac{f''\left(1\right)}{2},
$$
$$
\varepsilon\left(s\right)=\sum_{u=3}^{\infty}p_{u}\sum_{j=2}^{k-1}\sum_{v=1}^{j-1}\left(1-s^{v}\right), \ \ 0 \leq s <1.
$$
Thus,
\begin{align*}
1-f_{n}\left(0\right)
&= \left(1+\frac{Bn}{2}-\sum_{k=0}^{n-1}\delta\left(f_{k}\left(0\right)\right) \right)^{-1}\\
&= \left(1+\frac{Bn}{2} \right)^{-1}+\frac{1+o(1)}{\left(1+\frac{Bn}{2} \right)^{2}}\sum_{k=0}^{n-1}\delta\left(f_{k}\left(0\right)\right).
\end{align*}
Consequently,
\begin{align*}
h'(1)\sum_{j=1}^{n-1}(1-f_{j}(0))
&=h'(1)\sum_{j=1}^{n-1}\left(1+\frac{Bj}{2} \right)^{-1}
+\sum_{j=1}^{n-1}\frac{1+o(1)}{\left(1+\frac{Bj}{2} \right)^{2}}\sum_{k=0}^{j-1}\delta\left(f_{k}\left(0\right)\right)\\
&\ge \gamma\log n-C+c\sum_{j=1}^{n-1}\frac{1}{j^{2}}\sum_{k=0}^{j-1}\delta\left(f_{k}\left(0\right) \right).
\end{align*}
Combining this with \eqref{eq:F2}, we see that the convergence $L\left(n\right)\to 0$ will be proven if we show that the latter sum grows unboundedly. We first notice that
$$
\sum_{j=1}^{n-1}\frac{1}{j^{2}}\sum_{k=0}^{j-1}\delta\left(f_{k}\left(0\right) \right)
= \sum_{k=0}^{n-1}\delta\left(f_{k}\left(0\right) \right)\sum_{j=k+1}^{n-1}\frac{1}{j^{2}}
\geq \frac{1}{2}\sum_{k=1}^{n/2}\frac{1}{k}\delta\left(f_{k}\left(0\right) \right).
$$
It is immediate from the definition of $\delta\left(s\right)$ that
$$
\delta\left(s\right) \geq \varepsilon\left(s\right)-\alpha\left(1-s \right)\left(\alpha-\varepsilon\left(s\right) \right).
$$
Thus, we need to show that
\begin{equation}
\label{eq:28}
\sum_{k=1}^{\infty}\frac{\varepsilon\left(f_{k}\left(0\right)\right)}{k} = \infty.
\end{equation}
Changing the order of summation in the definition of $\varepsilon(s)$, we have
\begin{align*}
\varepsilon\left(f_{k}\left(0\right)\right)
& = \sum_{u=3}^{\infty}p_{u}\sum_{j=2}^{u-1}\sum_{v=1}^{j-1}\left(1-f_{k}^{v}\left(0\right)\right) \\
& =\sum_{j=2}^{\infty}\sum_{u=j+1}^{\infty}p_{u}\sum_{v=1}^{j-1}\left(1-f_{k}^{v}\left(0\right)\right)
=\sum_{j=2}^{\infty}K_{j}\sum_{v=1}^{j-1}\left(1-f_{k}^{v}\left(0\right)\right),
\end{align*}
where $K_{j}=\sum_{u=j+1}^{\infty}p_{u}$. Interchanging the order of the remaining two sums, we finally get
\begin{align*}
\varepsilon\left(f_{k}\left(0\right)\right)
=\sum_{v=1}^{\infty}\left(1-f_{k}^{v}\left(0\right)\right)\sum_{j=v+1}^{\infty}K_{j}
= \sum_{v=1}^{\infty}\left(1-f_{k}^{v}\left(0\right)\right)\overline{K_{v}},
\end{align*}
where $\overline{K_{v}}=\sum_{j=v+1}^{\infty}K_{j}$.
This representation implies that
\begin{align*}
 \sum_{k=1}^\infty\frac{1}{k}\varepsilon\left(f_{k}\left(0\right)\right)
& = \sum_{k=1}^\infty\frac{1}{k}
\sum_{v=1}^{\infty}\overline{K_{v}}\left(1-f_{k}^{v}\left(0\right)\right) \\
& \geq \sum_{k=1}^{\infty}\frac{1}{k}\sum_{v=1}^{k}\left(1-f_{k}^{v}\left(0\right)\right)\overline{K_{v}}
=\sum_{v=1}^{\infty} \overline{K_{v}}\sum_{k=1}^{v}\frac{1-f_{k}^{v}\left(0\right)}{k}.
\end{align*}
Noting that $f_{k}\left(0\right)$ is increasing and using
\eqref{eq:f-iterations}, we obtain
$$
\sum_{k=1}^{v}\frac{1-f_{k}^{v}\left(0\right)}{k}
\ge \left(1-f_{v}^{v}\left(0\right)\right)\sum_{k=1}^{v}\frac{1}{k}
\ge c\log v.
$$
Therefore,
$$
\sum_{k=1}^\infty\frac{1}{k}\varepsilon\left(f_{k}\left(0\right)\right)
\ge c \sum_{v=1}^{\infty} \overline{K_{v}}\log v.
$$
Now it remains to notice that the assumption
$\sum_{k=1}^{\infty} \left(k^{2}\log k\right)p_{k}=\infty$ is equivalent to
$\sum_{v=1}^{\infty} \overline{K_{v}}\log v=\infty$.
This concludes the proof of \eqref{eq:28}, which, in its turn, implies \eqref{eq:24}.
\end{proof}

To investigate the asymptotic behavior of lower deviations for $Y_{n}$, we need the following lemma, which deals with lower and upper bounds for $\pr\left(Y_{n}\leq k\right)$.
\begin{lemma}
\label{lemma5}
Assume that \eqref{eq:moments1} holds. Then there exist positive constants $c_{1}$ and $c_{2}$ such that
\begin{equation}
\label{eq:19}
c_{1}\prod_{j=k}^{n-1}h\left(f_{j}\left(0\right) \right) \leq \pr\left(Y_{n} \leq k\right) \leq c_{2}\prod_{j=k}^{n-1}h\left(f_{j}\left(0\right) \right)
\end{equation}
for all $1\le k\le n$.
\end{lemma}
\begin{proof}
We first prove the upper bound in \eqref{eq:19}.
By the Markov inequality,
$$
\pr\left(Y_{n} \leq k\right)
=\pr\left(s^{Y_{n}} \geq s^{k}\right)
\leq \frac{H_{n}\left(s\right)}{s^{k}},
\quad s\in(0,1].
$$
Letting $s=f_{k}\left(0\right)$, we get
$$
\pr\left(Y_{n} \leq k\right) \leq \frac{H_{n}\left(f_{k}\left(0\right)\right)}{\left(f_{k}\left(0\right)\right)^{k}}.
$$
Notice that \eqref{eq:f-iterations} implies that
$ \left(f_{k}\left(0\right)\right)^{k}\ge c_{3}>0$.
Therefore,
$$
\pr\left(Y_{n} \leq k\right) \leq \frac{H_{n}\left(f_{k}\left(0\right)\right)}{c_3}.
$$
Observe that
\begin{equation}
\label{eq:20}
 H_{n}\left(f_{k}\left(0\right)\right)= \prod_{j=0}^{n-1} h\left(f_{j}\left(f_{k}\left(0\right)\right) \right)=
\prod_{j=k}^{n+k-1} h\left(f_{j}\left(0\right) \right).
\end{equation}
To obtain the upper bound in \eqref{eq:19} it remains to notice that 
\begin{equation*}
\prod_{j=k}^{n+k-1} h\left(f_{j}\left(0\right)\right)
=
\prod_{j=k}^{n-1} h\left(f_{j}\left(0\right)\right)
\prod_{j=n}^{n+k-1} h\left(f_{j}\left(0\right)\right)
\le \prod_{j=k}^{n-1} h\left(f_{j}\left(0\right)\right).
\end{equation*}

To prove the lower bound in \eqref{eq:19}, we notice that if all immigrants from the first $n-k$ generations have no descendants at time $n$ then $Y_{n}$ has the same distribution as $Y_{k}$. Therefore,
$$ \pr\left(Y_{n} \leq k\right) \geq  \pr\left(Y_{n-k}=0\right)\pr\left(Y_{k} \leq k \right). $$
Due to the limit theorem \eqref{eq:gamma-limit}, there exists a positive number $c$ such that
$$ \pr\left(Y_{k} \leq k \right)\geq c>0 $$
uniformly in $k$.
Furthermore, it is easy to see that
$$
\pr\left(Y_{n-k}=0\right)
=\prod_{i=1}^{n-k}\pr \left(Z_{n-i}^{\left(i\right)}=0 \right)
= \prod_{j=k}^{n-1} h\left(f_{j}\left(0\right) \right).
$$
Thus, the lower bound is also proven.
\end{proof}
 \section{Estimates for concentration functions.}
The aim of this section is to provide some estimates for local probabilities of random variables $Y_{n}$ and $Z^{(i)}_{n-i}$. 

We begin with bounds for the concentration function
of $Y_n$.
\begin{lemma}
\label{lemma7}
For every $\varepsilon>0$, there exists $N=N\left(\varepsilon\right)$ such that for all $\left|s\right|\leq 1$,
$$
\left|H_{n}\left(s\right)\right| \leq \exp \left\{-\left(h'\left(1\right)-B\varepsilon\right) \sum_{j=N}^{n-1} \operatorname{Re}\left(1-f_{j}\left(s\right)\right)+\varepsilon \sum_{j=N}^{n-1}\left|\operatorname{Im}\left(1-f_{j}(s)\right)\right|\right\} .
$$
\end{lemma}
\begin{proof}
The proof is the same as the proof of Theorem in \cite{NV06}. For readers convenience and completeness, we provide the proof below.

Since $h\left(f_{j}\left(s\right)\right) \leq 1$, we obtain
\begin{equation}
\label{eq:eqH}
\left|H_{n}\left(s\right)\right| \leq \left|\prod_{j=N}^{n-1}h\left(f_{j}\left(s\right)\right) \right|=  \exp\left\{\sum_{j=N}^{n-1}\mathrm{Re} \log h\left(f_{j}\left(s\right)\right)\right\}.
\end{equation}
In view of \eqref{eq:17}, for $\left|s\right|\leq 1$ we have
\begin{align}
\label{eq:Re}
\nonumber
\mathrm{Re}\log h\left(f_{j}\left(s\right)\right)
& \leq -\left(h'\left(1\right)-\left|\mathrm{Re}\alpha\left(f_{j}\left(s\right)\right) \right| \right)\mathrm{Re}\left(1-f_{j}\left(s\right) \right) \\
& + \left|\mathrm{Im}\alpha\left(f_{j}\left(s\right) \right) \right| \left|\mathrm{Im}\left(1-f_{j}\left(s\right) \right)\right|.
\end{align}
Using the fact that $f_{j}\left(s\right)\to 1$ uniformly in $s$ from the unity disk, one can choose $N=N\left(\varepsilon\right)$ such that for all $j \geq N$,
$$
\left|\mathrm{Re}\alpha\left(f_{j}\left(s\right)\right)\right| \leq \varepsilon B, \ \
\left|\mathrm{Im}\alpha\left(f_{j}\left(s\right)\right) \right|\leq \varepsilon.
$$
Substituting the above bounds into \eqref{eq:Re} and recalling \eqref{eq:eqH}, we arrive at the conclusion of the lemma.
\end{proof}
\begin{lemma}
\label{lemma8}
Assume that \eqref{eq:moments1} holds. Then for each
$\varepsilon \in (0,2\gamma)$ there exists a constant
$c>0$ such that, for any $n \geq 1$ and $\left|t\right| \leq \pi / 2$,
$$
\left|H_n\left(e^{\mathrm{i}t}\right)\right| \leq c\left(n|t|\right)^{-2\gamma+\varepsilon}.
$$
In particular, there exists a constant $c^{*}$ such that, for every $n\ge1$ and all $\left|t\right| \leq \pi / 2$,
$$
\left|H_n\left(e^{\mathrm{i}t}\right)\right| \leq c^*\left(n|t|\right)^{-\theta},
$$
where $\theta=\min\{2\gamma,1/2\}$.
\end{lemma}
\begin{proof}
Consider the the same function $\psi=\psi\left(x\right)$ defined as in \cite{NV06}:
$$ \psi\left(x\right)=\frac{2\left(Bx+1\right)}{\left(Bx+1\right)^{2}+c^{2}}, \ \ x \geq 0, $$
where $c$ is a real number.
Replacing $B$ with $B/2$ in $\psi$, it is immediate from Lemma 7 in \cite{NV06} that
\begin{equation}
\label{eq:psi}
\left|\sum_{k=0}^{n-1}\psi\left(k\right)-\frac{2}{B}\log\left(1+\frac{\left(Bn/2+1\right)^{2}}{c^{2}} \right) \right| \leq 1.
\end{equation}
Further, repeating almost verbatim the arguments used in the proof of Lemma 8 from \cite{NV06}, one may check that for any $\varepsilon>0$, there exist $N$ and constants $a=a\left(\varepsilon,N \right)$ and $b=b\left(\varepsilon,N \right)$ such that for all $\left|t \right|\leq \pi/2$,
\begin{equation}
\label{eq:eqforRe}
\sum_{j=N}^{n-1}\mathrm{Re}\left(1-f_{j}\left(e^{\mathrm{i}t}\right) \right)\geq \frac{2\left(1-\varepsilon \right)}{B}\log\left(1+\left(Bn/2+1 \right)^{2}\tan^{2}\frac{t}{2}\right)-a,
\end{equation}
\begin{equation}
\label{eq:eqforIm}
\sum_{j=N}^{n-1}\left|\mathrm{Im}\left(1-f_{j}\left(e^{\mathrm{i}t}\right) \right)\right| \leq \frac{2\varepsilon}{B}\log\left(1+\left(Bn/2+1 \right)^{2}\tan^{2}\frac{t}{2}\right)+b.
\end{equation}
Using Lemma 8 and applying then \eqref{eq:psi}--\eqref{eq:eqforIm} and 
$\left|x\right|\leq \left|\tan x\right|$, we conclude that
$$
\left|H_{n}\left(e^{\mathrm{i}t}\right)\right| \leq c\left(\varepsilon \right)\left(n\left|t\right|\right)^{-2\left(\left(\gamma-2\varepsilon \right)\left(1-\varepsilon\right)-\frac{2\varepsilon^{2}}{B}\right)}.
$$
This gives the desired estimate.
\end{proof}
The next lemma is crucial for the proof of Theorem \ref{thm:main}.
\begin{lemma}
\label{lemma9}
Assume that \eqref{eq:moments1} holds. Then, for all $n,m \geq 1$,
$$
\pr\left(Y_{n}=m \right) \leq \frac{c}{m}.
$$
\end{lemma}
\begin{proof}
We make use of the following bound for the concentration function (see, for example, \cite{NV06}):
\begin{equation}
\label{eq:con.f.bound}
\sup_{x}\pr\left(\xi = x \right) \leq \left(\frac{96}{95}\right)^{2}\frac{1}{a}\int_{-a}^{a}\left|\varphi\left(t\right)\right|\mathrm{d}t,
\end{equation}
where $\varphi\left(t\right)$ is the characteristic function of the random variable $\xi$, and $a>0$.\\
Now let $\xi$ be a random variable with the generating function $H_{n}^{\prime}\left(s\right) / H_{n}^{\prime}\left(1\right)$. Choosing $a=\pi/2$, we conclude that there exists an absolute constant $c$ such that for any $n \geq 1$,
\begin{equation}
\label{eq:29}
\sup_{k} k \pr\left(Y_{n}=k\right) \leq c \int_{-\pi/2}^{\pi/2}\left|H_{n}^{\prime}\left(\mathrm{e}^{\mathrm{i} t}\right)\right|\mathrm{d}t.
\end{equation}
Observe that for $|s|\le 1$ one has
$$
H_{n}^{\prime}\left(s\right)=\sum_{j=0}^{n-1} h^{\prime}\left(f_{j}\left(s\right)\right) f_{j}^{\prime}\left(s\right) \prod_{l\neq j} h\left(f_{l}\left(s\right)\right) = H_{n}\left(s\right)\sum_{l=0}^{n-1}\frac{h^{\prime}
\left(f_{l}\left(s\right)\right)}
{h\left(f_{l}\left(s\right)\right)} f_{l}^{\prime}\left(s\right).
$$

We shall estimate integrals of $\left|H_{n}(e^{\mathrm{i}t})\frac{h^{\prime}
\left(f_{l}\left(e^{\mathrm{i}t}\right)\right)}
{h\left(f_{l}\left(e^{\mathrm{i}t}\right)\right)} f_{l}^{\prime}\left(e^{\mathrm{i}t}\right)\right|$ separately.
If $l=0$ then $f_{l}'(s)\equiv 1$ and the function
$\frac{h'(f_{l}(s))}{h(f_{l}(s))}$ is bounded. Estimating $|H_{n}\left(s\right)|$ by Lemma~\ref{lemma8}, we then get
\begin{align}
\label{eq:l=0}
\int_{1/n<|t|<\pi/2}\left|H_{n}\left(e^{\mathrm{i}t}\right)\frac{h^{\prime}\left(e^{\mathrm{i}t}\right)}
{h\left(e^{\mathrm{i}t}\right)} \right|\mathrm{d}t
\le C.
\end{align}

Assume now that $l\ge1$. Since
$\left|f'_{l}\left(e^{\mathrm{i}t}\right)\right| \leq 1$
and since $\frac{h^{\prime}
\left(f_{l}\left(e^{\mathrm{i}t}\right)\right)}
{h\left(f_{l}\left(e^{\mathrm{i}t}\right)\right)}$ is bounded, we get, applying Lemma~\ref{lemma8},
\begin{align}
\label{eq:integral1}
\nonumber
&\int_{1 / n<|t|< 1 / l}\left|H_{n}\left(e^{\mathrm{i}t}\right)\frac{h^{\prime}
\left(f_{l}\left(e^{\mathrm{i}t}\right)\right)}
{h\left(f_{l}\left(e^{\mathrm{i}t}\right)\right)} f_{l}^{\prime}\left(e^{\mathrm{i}t}\right)\right|\mathrm{d}t\\ &\hspace{2cm}\leq c^*\int_{1/n<|t|< 1/l}\left(n|t|\right)^{-\theta}\mathrm{d}t
\leq \frac{c^*}{1-\theta}
\frac{l^{\theta-1}}{n^{\theta}}.
\end{align}
For $|t|\ge1/l$, we use additionally the inequality
$$
\left|f'_{l}\left(e^{\mathrm{i}t}\right)\right| \leq \frac{c}{\left(l\left|t \right|\right)^{3/2}},
$$
which was derived in the proof of Lemma~9 from \cite{NV06}. As a result we have
\begin{align}
\label{eq:integral2}
\nonumber
&\int_{1 /l<|t|<\pi/2}\left|H_{n}\left(e^{\mathrm{i}t}\right)\frac{h^{\prime}
\left(f_{l}\left(e^{\mathrm{i}t}\right)\right)}
{h\left(f_{l}\left(e^{\mathrm{i}t}\right)\right)} f_{l}^{\prime}\left(e^{\mathrm{i}t}\right)\right|\mathrm{d}t\\ &\hspace{2cm}\leq c_1\int_{1 /l<|t|<\pi/2}(n|t|)^{-\theta}(l|t|)^{-3/2}\mathrm{d}t
\leq c_2
\frac{l^{\theta-1}}{n^{\theta}}.
\end{align}
Summing over $l$ the estimates in \eqref{eq:integral1} and \eqref{eq:integral2} and taking into account \eqref{eq:l=0}, we obtain
\begin{align*}
\int_{1/n<|t|<\pi/2}\left|H'_n(e^{\mathrm{i}t})\right|\mathrm{d}t
\le c.
\end{align*}
Using the elementary bound
$$
\left|H'_{n}\left(e^{\mathrm{i}t}\right)\right|\le h'(1)n,
$$
we get
\begin{align*}
\int_{|t|<1/n}\left|H'_{n}\left(e^{\mathrm{i}t}\right)\right|\mathrm{d}t
\le 2h'(1).
\end{align*}
Consequently,
$$
\int_{-\pi/2}^{\pi/2}\left|H_{n}^{\prime}\left(\mathrm{e}^{\mathrm{i} t}\right)\right|\mathrm{d}t\le c.
$$
Plugging this into \eqref{eq:29}, we complete the proof of the lemma.
\end{proof}
For the proof of our main result, we also need a concentration function bound for the Galton-Watson process $Z^{(l)}$, which starts with a random number of particles.
\begin{lemma}
\label{lemma10}
For all $1\le l<n$ one has
$$
\max_{k\ge1}\pr\left(Z^{(l)}_{n-l}=k\right)
\leq \frac{c}{(n-l)^{2}}.
$$
\end{lemma}
\begin{proof}
Let $Z_{n}$ denotes a critical Galton-Watson process. It has been shown in \cite{NV06} that
$$
\max_{k\ge1}\pr\left(Z_{n}=k|Z_{0}=1\right)\le\frac{c}{(n+1)^2}.
$$
This easily implies that
$$
\max_{k\ge1}\pr\left(Z_{n}=k|Z_{0}=j\right)\le\frac{cj}{(n+1)^{2}}
$$
for all $j\ge1$.
Therefore,
\begin{align*}
\pr\left(Z^{(l)}_{n-l}=k\right)
&=\sum_{j=1}^\infty q_{j}\pr(Z_{n-l-1}=k|Z_0=j)\\
&\le\frac{c}{(n-l)^2}\sum_{j=1}^\infty jq_{j}.
\end{align*}
Since the right-hand side does not depend on $k$ and since the expectation of $\{q_{j}\}$ is finite, we obtain the desired bound.
\end{proof}
\begin{lemma}
\label{lemma10'}
For all $1\le l<n$ one has
$$
\pr\left(Z^{(l)}_{n-l}=k\right) \le \frac{c}{k(n-l)}.
$$
\end{lemma}
\begin{proof}
The generating function of $Z^{(l)}_{n-l}$ equals
$h\left(f_{n-l-1}(s)\right)$. Therefore, the generating function of
the sequence $k\pr\left(Z^{(l)}_{n-l}=k\right)$ equals
$h'\left(f_{n-l-1}(s)\right)f'_{n-l-1}(s)$.
Applying \eqref{eq:con.f.bound} with $a=\pi/2$, we have
\begin{align*}
k\pr\left(Z^{(l)}_{n-l}=k\right)
&\le c\int_{-\pi/2}^{\pi/2}
|h'\left(f_{n-l-1}(e^{\mathrm{i}t})\right)f'_{n-l-1}(e^{\mathrm{i}t})|\mathrm{d}t \\
&\le ch'(1)\int_{-\pi/2}^{\pi/2}
|f'_{n-l-1}(e^{\mathrm{i}t})|\mathrm{d}t.
\end{align*}
It has been shown in the proof of Lemma 9 in \cite{NV06}
that the integral on the right-hand side is bounded by
$c/\left(n-l\right)$. Thus, the proof is finished.
\end{proof}
The next result is a simple generalisation of the local limit theorem for critical Galton-Watson processes proven in \cite{NV06}.
\begin{lemma}
\label{lem:llt}
Let $Z_{n}$ be a critical Galton-Watson process with finite variance. Let $g\left(s\right)$ be a generating function of $Z_{0}$.
If $g'(1)<\infty$ then
$$
\pr\left(Z_{n}=j\right)=\frac{4g'(1)}{B^{2}n^{2}}\exp\left\{-\frac{2j}{Bn}\right\}\left(1+o\left(1\right) \right)
$$
for all $j$ such that the ratio $j/n$ remains bounded.
\end{lemma}
\begin{proof}
 Let $\{g_{k}\}$ denote the coefficients of $g\left(s\right)$. Then we have
\begin{align*}
\pr\left(Z_n=j\right)
&=\sum_{k=1}^\infty g_{k}\pr\left(Z_{n}=j|Z_{0}=k\right)\\
&=\sum_{k=1}^\infty g_{k}\sum_{t=1}^{k}{k\choose t}
(1-f_n(0))^{t}f_{n}^{k-t}(0)
\pr\left(\sum_{i=1}^t\xi_i^{(n)}=j\right),
\end{align*}
where $\{\xi_{m}^{(n)}\}$ is a sequence of independent and identically distributed random variables with joint distribution
$$
\pr\left(\xi_{m}^{(n)}=j\right)=\pr\left(Z_{n}=j|Z_{0}=1,Z_{n}>0\right).
$$
Separating the summands corresponding to $t=1$, we have
\begin{align}
\label{eq:llt1}
\nonumber
\pr\left(Z_{n}=j\right)
&=\sum_{k=1}^\infty kg_k(1-f_n(0))^{k-1}
\pr(Z_n=j|Z_0=1)\\
&\hspace{0.5cm}
+\sum_{k=1}^\infty g_{k}\sum_{t=2}^k{k\choose t}
\left(1-f_n(0)\right)^{t}f_{n}^{k-t}(0)
\pr\left(\sum_{i=1}^t\xi_{i}^{(n)}=j\right).
\end{align}
According to Lemma 11 in \cite{NV06},
$$
\max_{j\ge1}\pr\left(\sum_{i=1}^t\xi_{i}^{(n)}=j\right)
\le\frac{c}{n\sqrt{t}}.
$$
Therefore,
\begin{align*}
&\sum_{k=1}^\infty g_k\sum_{t=2}^k{k\choose t}
\left(1-f_{n}(0)\right)^{t}f_{n}^{k-t}(0)
\pr\left(\sum_{i=1}^{t}\xi_i^{(n)}=j\right)\\
&\hspace{1cm}
\le\frac{c}{n}\sum_{k=1}^\infty g_{k}
\sum_{t=2}^k{k\choose t}(1-f_{n}(0))^{t}f_{n}^{k-t}(0)\\
&\hspace{1cm}
=\frac{c}{n}\sum_{t=2}^\infty \frac{1}{t!}(1-f_{n}(0))^{t}
\sum_{k=t}^\infty g_k\frac{k!}{(k-t)!}
f_{n}^{k-t}(0)\\
&\hspace{1cm}
=\frac{c}{n}\sum_{t=2}^\infty \frac{1}{t!}(1-f_{n}(0))^{t}
g^{(t)}(f_{n}(0))
=\frac{c}{n}\left(1-g(f_{n}(0))-g'(1)(1-f_{n}(0))\right).
\end{align*}
In the last step, we have used the Taylor series representation for $g$ at point $f_{n}\left(0\right)$.
The assumption $g'(1)<\infty$ implies that
$$
1-g(f_{n}(0))-g'(1)(1-f_{n}(0)
=o\left(1-f_{n}(0)\right)=o\left(n^{-1}\right).
$$
Consequently,
\begin{equation}
\label{eq:llt2}
\sum_{k=1}^\infty g_k\sum_{t=2}^k{k\choose t}
(1-f_{n}(0))^{t}f_{n}^{k-t}(0)
\pr\left(\sum_{i=1}^{t}\xi_i^{(n)}=j\right)
=o\left(n^{-2}\right)
\end{equation}
uniformly in $j$. It remains to notice that Theorem in \cite{NV06} implies that
\begin{align*}
\sum_{k=1}^\infty kg_{k}f_{n}^{k-1}(0)\pr\left(Z_{n}=j|Z_{0}=1\right)
&=g'(f_{n}(0))\pr\left(Z_{n}=j|Z_{0}=1\right)\\
=\frac{4g'(1)}{B^{2}n^{2}}\exp\left\{-\frac{2j}{Bn}\right\}\left(1+o\left(1\right) \right)
\end{align*}
uniformly in all $j$ such that the ratio $j/n$ remains bounded.
Plugging this relation and \eqref{eq:llt2} into \eqref{eq:llt1}, we get the desired result.
\end{proof}

 \section{Asymptotics for lower deviation probabilities:\\ Proof of Theorem~\ref{thm:main}.}
We start by proving \eqref{eq:main2}.
Let us first show that the case $Y_0=i$ with arbitrary $i\ge0$ can be reduced to the case $Y_n=0$. By the branching property,
\begin{align}
\label{eq:reduction}
\nonumber
\pr(Y_n=k|Y_0=i)
&=\sum_{j=0}^k\pr(Y_n=j|Y_0=0)\pr(Z_n=k-j|Z_0=i)\\
\nonumber
&=\pr(Y_n=k|Y_0=0)\left(f_n(0)\right)^i\\
&\hspace{1cm}
+\sum_{j=0}^{k-1}\pr(Y_n=j|Y_0=0)\pr(Z_n=k-j|Z_0=i).
\end{align}
Using Lemma~\ref{lem:llt}, we obtain 
\begin{align*}
\sum_{j=0}^{k-1}\pr(Y_n=j|Y_0=0)\pr(Z_n=k-j|Z_0=i)
\le C\frac{i}{n^2}\pr(Y_n\le k|Y_0=0).
\end{align*}
If Theorem~\ref{thm:main} is proven in the case $i=0$
then 
$$
\pr(Y_n\le k|Y_0=0)
\le C\frac{k^\gamma L(k)}{n^\gamma L(n)}
$$
and, consequently,
$$
\sum_{j=0}^{k-1}\pr(Y_n=j|Y_0=0)\pr(Z_n=k-j|Z_0=i)
=o\left(\pr(Y_n=k|Y_0=0)\right).
$$
Plugging this into \eqref{eq:reduction}, we conclude that
$$
\pr(Y_n=k|Y_0=i)\sim \pr(Y_n=k|Y_0=0)
$$
for every fixed $i$. 
From now on we shall assume that $Y_0=0$.
By the law of total probability, for every $k\ge1$,
\begin{equation}
\label{eq:34}
 \pr\left(Y_{n}=k\right)= \sum_{l=1}^{n}\pr\left(Y_{n}=k, \theta_{n}=l \right).
\end{equation}
 Furthermore, for every $l$ we have
\begin{align}
\label{eq:Y-theta}
\nonumber
\pr\left(Y_{n}=k, \theta_{n}=l \right)
&= \pr\left(Z_{n-l}^{\left(l\right)}+Y_{n-l}=k, Z_{n-l}^{\left(l\right)}>0 \right)\prod_{j=1}^{l-1}  \pr\left(Z_{n-j}^{\left(j\right)}=0\right)\\
&=\pr\left(Z_{n-l}^{\left(l\right)}+Y_{n-l}=k, Z_{n-l}^{\left(l\right)}>0 \right)\prod_{j=n-l+1}^{n-1}h(f_{j}(0)).
\end{align}

We split the sum in \eqref{eq:34} into three parts. More precisely, we fix some $\varepsilon\in(0,1)$ and consider separately the sums over $\left[1, n-\varepsilon^{-1}k\right)$, $ \left[n-\varepsilon^{-1}k, n-\varepsilon k\right)$ and $\left[n-\varepsilon k, n\right]$.

To estimate the sum over $\left[1, n-\varepsilon^{-1}k\right)$, we apply first Lemma \ref{lemma10} to get
\begin{align*}
\pr\left(Z_{n-l}^{\left(l\right)}+Y_{n-l}=k,
Z_{n-l}^{\left(l\right)}>0\right)
& = \sum_{m=1}^{k} \pr\left(Z_{n-l}^{\left(l\right)}=m\right) \pr \left(Y_{n-l}=k-m \right)  \\
& \leq \frac{c}{(n-l)^2} \sum_{m=1}^{k} \pr\left(Y_{n-l}=k-m \right) \\
& \leq \frac{c}{(n-l)^2}\pr\left(Y_{n-l}\leq k \right).
\end{align*}
Applying now Lemmas~\ref{lemma4} and \ref{lemma5}, we conclude that
\begin{align*}
\pr\left(Z_{n-l}^{\left(l\right)}+Y_{n-l}=k,
Z_{n-l}^{\left(l\right)}>0\right)
&\le c_1\frac{1}{(n-l)^2}
\frac{k^\gamma L(k)}{(n-l)^\gamma L(n-l)}\\
&\le c_2(1-h(f_{n-l}(0)))
\frac{k^\gamma L(k)}{(n-l)^{\gamma+1} L(n-l)}.
\end{align*}
Combining this with \eqref{eq:Y-theta}, we obtain
\begin{align*}
\sum_{l\leq n-\varepsilon^{-1}k}
\pr\left(Y_{n}=k, \theta_{n}=l \right)
\le c k^{\gamma} L\left(k\right)
\sum_{l\leq n-\varepsilon^{-1}k}\pr\left(\theta_{n}=l\right)
\frac{1}{(n-l)^{\gamma+1} L\left(n-l\right)}.
\end{align*}

It is immediate from \eqref{eq:theta-prob2} and Lemma~\ref{lemma4} that
$$ \pr\left(\theta_{n}=l \right) \leq c \frac{\left(n-l \right)^{\gamma-1} L\left(n-l \right)}{n^{\gamma}L\left(n\right)}.$$
This bound implies that
$$
\sum_{l\leq n-\varepsilon^{-1}k}\pr\left(\theta_{n}=l\right)
\frac{1}{(n-l)^{\gamma+1} L\left(n-l\right)}
\le c\frac{1}{n^{\gamma}L\left(n\right)}
\sum_{l\leq n-\varepsilon^{-1}k}(n-l)^{-2}
\le c\frac{\varepsilon}{kn^{\gamma}L\left(n\right)}
$$
and, consequently,
\begin{equation}
\label{eq:39}
\sum_{l\leq n-\varepsilon^{-1}k}
\pr\left(Y_{n}=k, \theta_{n}=l \right)
\le c\varepsilon\frac{k^{\gamma-1}L\left(k\right)}{n^{\gamma}L\left(n\right)}.
\end{equation}

Combining Lemmas \ref{lemma9} and \ref{lemma10'}, we have
\begin{align*}
 \pr\left(Z_{n-l}^{\left(l\right)}+Y_{n-l}=k,Z_{n-l}^{\left(l\right)}>0\right)
& =\sum_{m=1}^{k} \pr\left(Z_{n-l}^{\left(l\right)}=m\right) \pr \left(Y_{n-l}=k-m \right)\\
& =\sum_{m=1}^{k/2} \pr\left(Z_{n-l}^{\left(l\right)}=m\right) \pr \left(Y_{n-l}=k-m \right) \\
&\hspace{1cm} +\sum_{m=k/2}^{k} \pr\left(Z_{n-l}^{\left(l\right)}=m\right) \pr \left(Y_{n-l}=k-m \right) \\
&\le \frac{c}{k}\pr(Z_{n-l}^{\left(l\right)}>0)
+\frac{c}{k(n-l)}\pr(Y_{n-l}\le m/2)\\
& \leq \frac{c}{k(n-l)}\le \frac{c(1-h(f_{n-l}(0)))}{k}.
\end{align*}
Applying this bound to the right-hand side in \eqref{eq:Y-theta} and taking into account \eqref{eq:theta-prob2}, we obtain
\begin{align*}
\sum_{l \geq n-\varepsilon k}
\pr\left(Y_{n}=k, \theta_{n}=l \right)
&\le \frac{c}{k}\sum_{l \geq n-\varepsilon k}
(1-h(f_{n-l}(0))\prod_{j=n-l+1}^{n-1}h(f_j(0))\\
&=\frac{c}{k}\sum_{l \geq n-\varepsilon k}\pr(\theta_n=l)\\
&=\frac{c}{k}\left(\pr(\theta_n>n-\varepsilon k)-
\pr(\theta_n>n)\right)\\
& = \frac{c}{k} \left(\prod_{j=1}^{n-\varepsilon k} h\left(f_{n-j}\left(0\right) \right)-\prod_{j=1}^{n} h\left(f_{n-j}\left(0\right) \right)\right).
\end{align*}
Recalling now Lemma~\ref{lemma4}, we conclude that
\begin{equation}
\label{eq:44}
\sum_{l \geq n-\varepsilon k}
\pr\left(Y_{n}=k, \theta_{n}=l \right)
 \leq c\varepsilon^{\gamma}\frac{k^{\gamma-1}L\left(n\right)}{n^{\gamma}L\left(n\right)}.
\end{equation}
Thus it remains to consider the sum over
the interval $I_{k}:=[n-k/\varepsilon,n-\varepsilon k]$.
Applying Lemma~\ref{lem:llt}, we have, uniformly in $l\in I_{k}$ and $m\le k$,
$$ \pr\left(Z_{n-l}^{\left(l\right)}=m\right)= h'\left(1\right)\frac{4}{B^{2}(n-l)^{2}}\exp\left\{-\frac{2m}{B(n-l)}\right\}+o\left(\frac{1}
{(n-l)^{2}}\right).  $$
Therefore,
\begin{align*}
\pr\left(Z_{n-l}^{\left(l\right)}+Y_{n-l}=k,Z_{n-l}^{\left(l\right)}>0\right)
&=
\sum_{m=1}^{k} \pr\left(Z_{n-l}^{\left(l\right)}=m\right) \pr \left(Y_{n-l}=k-m \right)\\
&=\frac{4h'(1)+o\left(1\right)}{B^2(n-l)^{2}}
\mathbf{E}\left[e^{-2(k-Y_{n-l})/B(n-l)}; Y_{n-l} \leq k \right].
\end{align*}
According to \eqref{eq:gamma-limit},
\begin{align*}
 \mathbf{E}
 \left[e^{2Y_{n-l})/B(n-l)}; Y_{n-l} \leq k \right]
 &=\left(1+o\left(1\right)\right)\int_{0}^{2k/B(n-l)}e^{u}\frac{1}{\Gamma(\gamma)}
 u^{\gamma-1}e^{-u}du\\
 &=\frac{1+o\left(1\right)}{\Gamma(\gamma+1)}
 \left(\frac{2k}{B(n-l)}\right)^\gamma.
\end{align*}
Consequently,
\begin{equation}
\label{eq:lower1}
\pr\left(Z_{n-l}^{\left(l\right)}+Y_{n-l}=k,Z_{n-l}^{\left(l\right)}>0\right)
=\frac{h'(1)+o(1)}{\Gamma(\gamma+1)}\frac{1}{k^2}
\left(\frac{2k}{B(n-l)}\right)^{\gamma+2}
e^{-2k/B(n-l)}.
\end{equation}
We know from Lemma~\ref{lemma4} that
\begin{equation}
\label{eq:lower2}
\prod_{j=n-l+1}^{n-1}h\left(f_{j}(0)\right)
\sim\frac{(n-l)^{\gamma} L\left(n-l\right)}{n^{\gamma} L\left(n\right)}.
\end{equation}
Plugging \eqref{eq:lower1} and \eqref{eq:lower2} into
\eqref{eq:Y-theta} and summing over $l$, we deduce
\begin{align}
\label{eq:lower3}
\nonumber
\sum_{l\in I_k}
\pr\left(Y_{n}=k, \theta_{n}=l \right)
&=\frac{h'(1)+o(1)}{\Gamma(\gamma+1)n^{\gamma} L\left(n\right)}
\frac{1}{k^2}\left(\frac{2k}{B}\right)^{\gamma+2}
\sum_{l\in I_{k}}\frac{L\left(n-l\right)}{(n-l)^2}e^{-2k/B(n-l)}\\
&=\frac{h'(1)+o(1)}{\Gamma(\gamma+1)}\left(\frac{2}{B}\right)^{\gamma+2}\frac{k^{\gamma} L\left(k\right)}{n^{\gamma} L\left(n\right)}
\sum_{l\in I_{k}}\frac{e^{-2k/B(n-l)}}{(n-l)^2}.
\end{align}
Approximating the sum by the integral, we obtain
\begin{align*}
\sum_{l\in I_{k}}\frac{e^{-2k/B(n-l)}}{(n-l)^2}
&=\frac{1}{k}\sum_{\varepsilon k\le j\le k/\varepsilon}
\left(\frac{k}{j}\right)^{2}e^{-2k/Bj}\frac{1}{k}\\
&=\frac{1+o(1)}{k}\int_{\varepsilon}^{1/\varepsilon}
u^{-2}e^{-2/Bu}\mathrm{d}u\\
&=\frac{B+o(1)}{2k}\left(e^{-2\varepsilon/B}
-e^{-2/B\varepsilon}\right).
\end{align*}
Plugging this into \eqref{eq:lower3} and recalling that
$\gamma=2h'(1)/B$, we have
\begin{align}
\label{eq:lower4}
\sum_{l\in I_{k}}
\pr\left(Y_{n}=k, \theta_{n}=l \right)
=\frac{1+o(1)}{\Gamma(\gamma)}
\left(\frac{2}{B}\right)^{\gamma}
\frac{k^{\gamma-1} L\left(k\right)}{n^{\gamma} L\left(n\right)}
\left(e^{-2\varepsilon/B}-e^{-2/B\varepsilon}\right).
\end{align}
Combining now \eqref{eq:39}, \eqref{eq:44}, \eqref{eq:lower4} and letting $\varepsilon \to 0$,
we get \eqref{eq:main2}.

The proof of \eqref{eq:main1} is very similar and, in some parts, even simpler. For that reasons, we omit it.

To conclude the proof of Theorem~\ref{thm:main}, it remains to establish \eqref{eq:main3}. To this end, we shall use again
\eqref{eq:34}. Since \eqref{eq:39} is valid for all $k$ it remains to consider the case when $l\ge n-k/\varepsilon$. In view of \eqref{eq:Y-theta}, we have to determine the asymptotic behavior of
$$
\sum_{m\le k/\varepsilon}
\pr\left(Z_{m}^{\left(l\right)}+Y_{m}=k, Z_{m}^{\left(l\right)}>0 \right)\prod_{j=m+1}^{n-1}h\left(f_{j}(0)\right).
$$
Since this sum can be rewritten in the following way
\begin{align*}
&\sum_{m\le k/\varepsilon}
\pr\left(Z_{m}^{\left(l\right)}+Y_{m}=k, Z_{m}^{\left(l\right)}>0 \right)\prod_{j=m+1}^{n-1}h\left(f_{j}(0)\right)\\
&\hspace{1cm}=\prod_{j=0}^{n-1}h\left(f_{j}(0)\right)
\sum_{m\le k/\varepsilon}\frac{
\pr\left(Z_{m}^{\left(l\right)}+Y_{m}=k, Z_{m}^{\left(l\right)}>0 \right)}{\prod_{j=0}^{m}h\left(f_{j}(0)\right)},
\end{align*}
we infer from Lemma~\ref{lemma3} that
\begin{align*}
 \sum_{l> n-\varepsilon^{-1}k}
\pr\left(Y_{n}=k, \theta_{n}=l \right)
\sim \frac{1}{n^{\gamma}L\left(n\right)}
\sum_{m\le k/\varepsilon}\frac{
\pr\left(Z_{m}^{\left(l\right)}+Y_{m}=k, Z_{m}^{\left(l\right)}>0 \right)}{\prod_{j=0}^{m}h\left(f_{j}(0)\right)}.
\end{align*}
Combining this with \eqref{eq:39}, we get the asymptotic relation for every fixed $k$. This completes
the proof of Theorem~\ref{thm:main}.

\end{document}